\documentclass[12pt]{amsart}
\usepackage{amsmath}
\usepackage{amsfonts}
\usepackage{amssymb}
\usepackage{amscd}
\usepackage{graphicx}
\usepackage[abbrev,alphabetic]{amsrefs}
\usepackage{hyperref}
\RequirePackage[dvipsnames,usenames]{color}
\usepackage{stmaryrd}
\usepackage{mathtools}
\usepackage{booktabs}
\usepackage{multirow}
\usepackage{stmaryrd} 
\newtagform{tiny}{\tiny(}{)}
\usepackage{tikz}
\usepackage{tikz-3dplot}

\makeatletter
\@addtoreset{section}{part}
\makeatother

\usepackage{mathtools}
\usepackage{hyperref}
\usepackage[margin=1.25in]{geometry}

\usepackage{amsthm}
\usepackage{comment}
\usepackage[all,cmtip]{xy}
\usepackage{tikz-cd}
\usetikzlibrary{cd}
\usepackage{mathtools}
\usepackage[all]{xy}

\makeatletter
\def\@tocline#1#2#3#4#5#6#7{\relax
  \ifnum #1>\c@tocdepth 
  \else
    \par \addpenalty\@secpenalty\addvspace{#2}%
    \begingroup \hyphenpenalty\@M
    \@ifempty{#4}{%
      \@tempdima\csname r@tocindent\number#1\endcsname\relax
    }{%
      \@tempdima#4\relax
    }%
    \parindent\z@ \leftskip#3\relax \advance\leftskip\@tempdima\relax
    \rightskip\@pnumwidth plus4em \parfillskip-\@pnumwidth
    #5\leavevmode\hskip-\@tempdima
      \ifcase #1
       \or\or \hskip 1em \or \hskip 2em \else \hskip 3em \fi%
      #6\nobreak\relax
    \hfill\hbox to\@pnumwidth{\@tocpagenum{#7}}\par
    \nobreak
    \endgroup
  \fi}
\makeatother

\newsavebox{\pullback}
\sbox\pullback{%
\begin{tikzpicture}%
\draw (0,0) -- (1ex,0ex);%
\draw (1ex,0ex) -- (1ex,1ex);%
\end{tikzpicture}}

\newsavebox{\pullbackdl}
\sbox\pullbackdl{%
\begin{tikzpicture}%
\draw (-1ex,0ex) -- (0ex,0ex);%
\draw (0ex,-1ex) -- (0ex,0ex);%
\end{tikzpicture}}

\newsavebox{\pushoutdr}
\sbox\pushoutdr{%
\begin{tikzpicture}%
\draw (-1ex,-1ex) -- (-1ex,0ex);%
\draw (-1ex,0ex) -- (0ex,0ex);%
\end{tikzpicture}}

\newcommand{\rdown}[1]{\lfloor #1 \rfloor}

\renewcommand{\P}{\mathbb{P}}
\newcommand{\Z}{\mathbb{Z}}
\newcommand{\Q}{\mathbb{Q}}

\newcommand{\R}{\mathbb{R}}
\newcommand{\F}{\mathbb{F}}

\newcommand{\MO}{\mathcal{O}}

\newcommand{\red}{\mathrm{red}}

\newcommand{\wt}{\widetilde}
\newcommand{\ol}{\overline}

\DeclareMathOperator{\pr}{pr}

\DeclareMathOperator{\Supp}{Supp}
\DeclareMathOperator{\Spec}{Spec}

\DeclareMathOperator{\Pic}{Pic}
\DeclareMathOperator{\NE}{NE}
\DeclareMathOperator{\Ex}{Ex}

\DeclareMathOperator{\mult}{mult}

\theoremstyle{plain}
\newtheorem{theorem}{Theorem}[section]
\newtheorem{thm}[theorem]{Theorem}

\newtheorem{prop}[theorem]{Proposition}

\newtheorem{lem}[theorem]{Lemma}

\newtheorem*{claim*}{Claim}
\newtheorem{step}{Step}

\newtheorem{theoremA}{Theorem}

\theoremstyle{definition}

\newtheorem{ex}[theorem]{Example}

\newtheorem{nota}[theorem]{Notation}

\newtheorem*{setup*}{Setup}

\theoremstyle{remark}
\newtheorem{rem}[theorem]{Remark}



\numberwithin{equation}{theorem}

\title[Mumford vanishing for threefolds in positive characteristic]{Mumford vanishing for threefolds in positive characteristic}

\author{Tatsuro Kawakami}
\address{Graduate School of Mathematical Sciences, 
The University of Tokyo, 
3-8-1 Komaba, Meguro-ku, Tokyo 153-8914, JAPAN} 
\email{kawakami@ms.u-tokyo.ac.jp}
\author{Hiromu Tanaka} 
\address{Department of Mathematics, 
Graduate School of Science, 
Kyoto University, 
Kyoto 606-8502, JAPAN} 
\email{tanaka.hiromu.7z@kyoto-u.ac.jp}

\begin{document}

\begin{abstract}
Let $X$ be a projective klt threefold in characteristic $p>5$ and 
let $L$ be a nef Cartier divisor on $X$. 
We show that $H^1(X, -L)=0$ for the following two cases: 
(1) $K_X$ is not big and $L$ is big; 
(2) $-K_X$ is nef and $L$ is of numerical dimension two. 
\end{abstract}

\maketitle

\setcounter{tocdepth}{2}

\tableofcontents

\section{Introduction}

Vanishing theorems play a central role in algebraic geometry,
providing powerful tools to study the geometry and cohomology of algebraic varieties. 
In characteristic zero, classical results such as the Kodaira vanishing theorem
and the Kawamata--Viehweg vanishing theorem are fundamental in the minimal model program.
However, in positive characteristic, these vanishing theorems are known to fail in general,
which makes it important to identify situations in which similar statements still hold.

In this direction, vanishing results for surfaces are well established
under additional assumptions.
More precisely, if $p>3$ and
$X$ is a projective normal surface such that $K_X$ is not big,
then $H^1(X, -L)=0$  for any nef and big Cartier divisor $L$ \cite[Theorem 3(a)]{Muk13}.
In this article, we prove the following analogous vanishing theorem for threefolds.

\begin{theoremA}[Theorem \ref{t H^1 nefbig}]\label{intro Mumford}
Let $X$ be a projective normal threefold over a field of characteristic $p>5$.
Assume that $K_X$ is $\Q$-Cartier and not big.
Let $L$ be a nef and big Cartier divisor.
Then $H^1(X, -L)=0$.
\end{theoremA}

\begin{rem}
\begin{enumerate}
\item 
The same assertion as in Theorem \ref{intro Mumford} does not hold in characteristic two,
even if the base field is algebraically closed (Theorem \ref{t Mukai p=2}).
\item 
We cannot drop the assumption that $K_X$ is not big, 
as otherwise there exists a counterexample in arbitrary characteristic \cite[Theorem 2]{Muk13}.
\item 
The above result is known in the case when $p>7$, $-K_X$ is nef, and $L$ is ample \cite[Theorem 1.6]{PW22}.
\end{enumerate}
\end{rem}

Mumford established the following vanishing theorem \cite[Theorem 2 in page 95]{Mum67}
\[
H^1(X, -L)=0
\]
for a projective normal variety $X$ in characteristic zero and a semi-ample Cartier divisor $L$
satisfying $\nu(X, L) \geq 2$. 
It is natural to ask under what assumptions the same conclusion holds 
for threefolds in positive characteristic. 
Let $S$ be a smooth projective surface with an ample Cartier divisor $L_S$ such that $H^1(S, L_S) \neq 0$.
Then $H^1(X, -L) \neq 0$ for $X := S \times \P^1$ and $L := \pr_1^*L_S$.
Therefore, even if we impose the condition $\kappa(X, K_X) = -\infty$,
Mumford vanishing fails for smooth threefolds in arbitrary positive characteristic. 
In this article, we establish Mumford vanishing for   threefolds with nef anti-canonical divisors in characteristic $p>5$.

\begin{theoremA}[Theorem \ref{t nu=2 Mumford}]\label{intro nu=2}
Let $X$ be a projective klt threefold over an algebraically closed field of characteristic $p>5$ such that $-K_X$ is  nef.
Let $L$ be a nef Cartier divisor satisfying $\nu(X, L) \geq 2$.
Then $H^1(X, -L)=0$.
\end{theoremA}

As an application, we obtain the following result for weak Fano threefolds.

\begin{theoremA}[Theorem \ref{t irreg}]
Let $X$ be a smooth weak Fano threefold over an algebraically closed field of characteristic $p>5$.
Then the following hold:
\begin{enumerate}
\item 
$H^i(X, \MO_X)=0$ for every $i>0$.
\item 
$X$ is simply connected, i.e., if $\pi: Y \to X$ is a finite \'etale morphism from a normal threefold $Y$, then $\pi$ is an isomorphism.
\item 
$\Pic X \simeq \Z^{\oplus \rho(X)}$.
\end{enumerate}
\end{theoremA}

\subsection{Overview of the proof of Theorem \ref{intro Mumford}}

Let $X$ and $L$ be as in Theorem \ref{intro Mumford}. 
Suppose that  $H^1(X, -L) \neq 0$. 
It suffices to derive a contradiction. 
Taking a desingularisation, we may assume that $X$ is regular. 
It is known that $H^1(X, -mL)=0$ for $m \gg 0$, 
and hence we may assume that $H^1(X, -pL)=0$. 
Then we can find a 
 finite inseparable surjective morphism  
\[
\varphi: X' \to X 
\]
of degree $p$ such that $X'$ is a projective Gorenstein threefold  and $K_{X'} \sim \varphi^*(K_X -(p-1)L)$. 
Taking the normalisation $\nu : X'' \to X'$ of $X'$, 
we get 
$K_{X''} +D \sim \nu^*K_{X'}$ for some effective divisor $D$ on $X''$. 
For the induced composite morphism 
\[
\psi: X'' \xrightarrow{\nu} X' \xrightarrow{\varphi} X, 
\]
it holds that 
\[
K_{X''} +D+ 5\psi^*L +(p-6)\psi^*L  \sim \nu^*(K_{X'} +(p-1) \varphi^* L) 
\sim \nu^*\varphi^*K_X =\psi^*K_X. 
\]
Since $K_{X''} + 5\psi^*L$ is pseudo-effective  (Theorem \ref{intro Fujita}), 
it follows from $p-6>0$ that $K_{X''} +D+ 5\psi^*L +(p-6)\psi^*L$ is big. 
Then $\psi^*K_X$ is big, and hence so is $K_X$. This is absurd.

\begin{theoremA}\label{intro Fujita}
Let $X$ be a projective normal threefold over a  field  of characteristic $p>5$ such that $K_X$ is $\Q$-Cartier. 
Take a nef and big Cartier divisor $L$ on $X$. 
Then $K_X+5L$ is pseudo-effective. 
\end{theoremA}

\subsection{Overview of the proof of Theorem \ref{intro nu=2}}

Let $X$ and $L$ be as in Theorem \ref{intro nu=2}. 
By Theorem \ref{intro Mumford}, we may assume $\nu(X, L)=2$. 
Suppose that $H^1(X, -L) \neq 0$. 
After possibly replacing $L$ by $p^eL$ for some integer $e \geq 0$, 
the same argument as in Theorem \ref{intro Mumford} 
allows us to construct a finite inseparable surjective morphism 
\[
\psi: Y \to X 
\]
of degree $p$ and an effective Weil divisor $D$ on $Y$ such that 
$Y$ is a projective normal threefold and 
\[
K_Y + D + (p-1)L_Y \sim \psi^*K_X 
\]
for  $L_Y := \psi^*L$. 
If $K_Y + 5L_Y$ is pseudo-effective, then this linear equivalence 
contradicts the assumption that $-K_X$ is nef. 
Take a desingularisation $\mu : Y' \to Y$ and run a $(K_{Y'} + 5\mu^*L_Y)$-MMP: $Y' \dashrightarrow Y''$. 
For simplicity, we assume that $Y = Y' = Y''$. 
Then there exists a $(K_Y + 5L_Y)$-negative Mori fibre space $\rho : Y \to T$, which is also $K_Y$-negative. 

The remaining proof consists of the following four steps:
\begin{enumerate}
\item $L$ is EWM. 
\item For the contraction $f: X \to S$ induced by $L$, 
$-K_X$ is $f$-nef and $f$-big. 
\item $L$ is semi-ample and descends to an ample Cartier divisor $L_S$ on $S$. 
\item $H^1(X, -L) = 0$. 
\end{enumerate}

(1) We say that a nef Cartier divisor $L$ is {\em EWM} if there exists a morphism 
$\pi: X \to S$ to a proper algebraic space $S$ with $\pi_*\MO_X = \MO_S$ such that, for any integral closed subscheme $V \subset X$, we have $\dim V > \dim \pi(V)$ if and only if 
$L^{\dim V} \cdot V = 0$. 
This notion, introduced by Keel in \cite{Kee99}, 
is similar to but weaker than semi-ampleness. 

Since $\rho : Y \to T$ is a $(K_Y+5L_Y)$-Mori fibre space,
we  obtain $L_Y \sim \rho^*L_T$ for some nef and big Cartier divisor $L_T$ on $T$. 
Hence $L_T$ is EWM, and therefore so are $L_Y (= \psi^*L)$ and $L$. 

(2) 
Note that $S$ is a two-dimensional normal irreducible proper algebraic space. 
In particular, after removing finitely many closed points, 
$S$ becomes a scheme, hence a normal surface. 
Since $-K_X$ is nef, it is automatically $f$-nef. 
It suffices to show that general fibres of $f$ are isomorphic to $\mathbb{P}^1$. 
Otherwise, general fibres would be smooth elliptic curves 
since $p > 3$ and $-K_X$ is $f$-nef. 

On the other hand, let $\pi_Y : Y \to S_Y$ be the contraction induced by the EWM nef Cartier divisor $\psi^*L$. 
Then general fibres of $\pi_Y$ are isomorphic to $\mathbb{P}^1$, since they coincide with general fibres of 
the $K_Y$-Mori fibre space $\rho : Y \to T$. 
Hence we can find a purely inseparable cover 
$X_s \to Y_t$ between 
a general fibre $X_s$ of $\pi$ (a smooth elliptic curve) 
and a general fibre $Y_t =\P^1$ of $\pi_Y$, which is absurd. 

(3) 
In order to show that $L$ is semi-ample, it suffices to prove that a suitable multiple of $L$ descends to $S$ after  taking the base change by an \'etale cover $S' \to S$. 
This follows from (2) and the relative base point free theorem. 
In particular, $S$ is a projective normal surface. 
Again by the relative base point free theorem, $L$ descends to $S$. 

(4) We have an exact sequence 
\[
0 \to H^1(S, \mathcal{O}_S(-L_S)) \to H^1(X, \mathcal{O}_X(-L)) \to H^0(S, R^1f_*\mathcal{O}_X(-L))
\]
arising from the corresponding Leray spectral sequence. 
By (2), we have 
$R^1f_*\mathcal{O}_X(-L) =0$. 
Hence it suffices to show that $H^1(S, \mathcal{O}_S(-L_S)) = 0$. 
Since $L_S$ is nef and big, it is enough to show that $K_S$ is not pseudo-effective, 
which follows from \cite{Eji19} and the nefness of $-K_X$.

\subsection{Overview of the proof of Theorem \ref{intro Fujita}}

By standard argument, we may assume that the base field is algebraically closed and $X$ is $\Q$-factorial and terminal. 
Suppose that $K_X+5L$ is not pseudo-effective. 
Running a $(K_X+5L)$-MMP, the problem is reduced to the case when 
there is a $(K_X+5L)$-Mori fibre space $\pi : X \to B$. 
If $\dim B>0$, then it is easy to derive a contradiction by taking the generic fibre of $\pi$. 
Assume that $\dim B=0$, i.e., 
$\rho(X)=1$ and $-K_X$ is ample. 
In this case, $L$ is ample. 

Set $A := -(K_X+5L)$, which is an ample Weil divisor. 
Fix a smooth point $P \in X$. 
By the same argument as in Shokurov's non-vanishing theorem \cite[\S 3.5]{KM98}, we can find an effective $\Q$-divisor $E$ on $X$ satisfying $E \sim_{\Q} 2L + (1/2)A$ and $\mult_P E >2$. 
Taking the log canonical threshold $\lambda$ of $(X, E)$, one of the following holds: 
\begin{enumerate}
\item[(I)] $(X, E)$ is klt and $\mult_P E >2$. 
\item[(II)] $(X, \lambda E)$ is lc but not dlt for some $0 < \lambda \leq 1$. 
\item[(III)] $(X, \lambda E)$ dlt but not klt for some $0 < \lambda \leq 1$. 
\end{enumerate}
For the the case (III), we shall derive a contradiction by applying  the adjunction for a prime divisor contained $S$ in $\rdown{\lambda E}$. 
For the cases (I) and (II), we take a suitable birational modification $f : X' \to X$ such that $(X', \Delta')$ is a $\Q$-factorial dlt pair and $K_{X'}+ \Delta' +sf^*L \equiv 0$ 
for some rational number $s>3$. 
Such a birational morphism $f$ is obtained by taking a blowup at $P$ 
(resp.\ a dlt modification of $(X, \lambda E)$) for the case (I) (resp.\ (II)). 
Since $K_{X'}+ \Delta' +(s-\epsilon)f^*L$ is not nef for $0 < \epsilon \ll 1$, 
we can find a $(K_{X'}+ \Delta' +(s-\epsilon)f^*L)$-negative extremal ray $R$. 
For its contraction $g:X' \to X''$,  we shall derive a contradiction by using the fact that $f^*L$ is $g$-ample and 
$s-\epsilon >3$. 
Essentially, this final step is also reduced to the lower dimensional case by 
using Shokurov's pl-MMP for the birational case and 
taking the generic fibre for the Mori-fibre-space case. 

\begin{rem}
By using the well-known bound $6 (=2 \dim X)$ of the length of extremal rays (cf.\ \cite[Theorem 1.1]{BW17}), 
we can easily prove that $K_X+6L$ is pseudo-effective under the same assumption as in 
Theorem \ref{intro Fujita}. 
Although our result (Theorem \ref{intro Fujita}) is slightly better, 
it is conjecturally expected that $K_X + 4L$ is pseudo-effective. 
\end{rem}


\medskip
\noindent {\bf Acknowledgements.}
Kawakami was supported by JSPS KAKENHI Grant number JP24K16897 and by Inamori Foundation.
Tanaka was supported by JSPS KAKENHI Grant numbers JP22H01112 and JP23K03028.

\section{Preliminaries}\label{s prelim}

\subsection{Notation}
\label{ss:notation}

In this subsection, we summarise notation and basic definitions used in this article. 
\begin{enumerate}
\item Throughout the paper, $p$ denotes a prime number and we set $\F_p \coloneqq \Z/p\Z$. 
Unless otherwise specified, we work over an algebraically closed field $k$ of characteristic $p>0$. 
\item We say that $X$ is a {\em variety} (over a field $\kappa$) if 
    $X$ is an integral scheme 
    that is separated and of finite type over $\kappa$. 
    We say that $X$ is a {\em curve} (resp. {\em surface}, resp. {\em threefold}) 
    if $X$ is a variety of dimension one (resp. two, resp. three). 
\item For a variety $X$, 
we define the {\em function field} $K(X)$ of $X$ 
as the stalk $\MO_{X, \xi}$ at the generic point $\xi$ of $X$. 
\item 
For the definition of singularities in minimal model program (klt, lc, etc), we refer to \cite{KM98} and \cite{Kol13}. 
\end{enumerate}

\subsection{Birational descent for numerically trivial divisors}

\begin{lem}\label{l surface 3L}
Let $(S, \Delta_S)$ be a two-dimensional projective lc pair. 
Let $g: S \to T$ be a morphism to a projective variety $T$. 
Take a $g$-ample Cartier divisor $L$ on $S$. 
Then $K_S+\Delta_S +3L$ is  $g$-nef. 
\end{lem}

\begin{proof}
See \cite[Proposition 6.3]{Tan14}. 
\end{proof}

\begin{prop}\label{p surface Fujita imperf}
Let $\kappa$ be a field. 
Let $S$ be a projective normal surface over $\kappa$ such that $K_S$ is $\Q$-Cartier. 
Take a nef and big Cartier divisor $L$. 
Then $K_S+ 3L$ is pseudo-efffective. 
\end{prop}

\begin{proof}
By \cite[Theorem 1.1]{Tan18b}, we may assume that $\kappa$ is an algebraically closed field. 
Taking the minimal resolution, we may assume that $S$ is smooth. 

Suppose that $K_S+ 3L$ is not pseudo-effective. 
We run a $(K_S+3L)$-MMP: 
\[
S =: S_0 \to S_1 \to \cdots \to S_{\ell}, 
\]
where the end result $S_{\ell}$ has a Mori fibre space $\pi : S_{\ell} \to B$ such that $(K_{S_{\ell}} + 3L_{S_{\ell}})$ is $\pi$-anti-ample. 
Note that this sequence is a $K_S$-MMP, and hence each $S_i \to S_{i+1}$ is a contraction of a $(-1)$-curve. 
Moreover, either  $\pi : S_{\ell} \to B$ is a $\P^1$-bundle or 
$S_{\ell} = \P^2$. 
In any case, $(K_{S_{\ell}} + 3L_{S_{\ell}})$ cannot be $\pi$-anti-ample, 
because $3L_{S_{\ell}}$ is $\pi$-ample. 
\end{proof}

\begin{thm}\label{t nume triv descent}
Assume $p>5$. 
Let $f: X \to Y$ be a projective morphism between quasi-projective varieties with $f_*\MO_X = \MO_Y$ and $\dim Y>0$. 
Assume that there exists an effective $\Q$-divisor $\Delta$ on $X$ such that 
\begin{enumerate}
\item $(X, \Delta)$ is a three-dimensional klt pair, and 
\item $-(K_X+\Delta)$ is $f$-nef and $f$-big. 
\end{enumerate}
Take a Cartier divisor $L$ on $X$ satisfying $L \equiv_f 0$. 
Then there exists a Cartier divisor $M$ on $Y$ such that $L \sim f^*M$.
\end{thm}
\begin{proof}
The assertion follows from \cite[Theorem 1.3 and Remark 1.4]{Ber21} 
and  \cite[Theorem 1.1]{ABL22}. 
\end{proof}

\begin{prop}\label{p birat descent}
Assume $p>5$. 
Let $(Z, \Delta_Z)$ be a three-dimensional quasi-projective klt pair. 
Let $f: X \to Z$ be a projective birational morphism from a normal threefold $X$. 
Take a Cartier divisor $L$ on $X$ satisfying $L \equiv_f 0$. 
Then $L \sim f^*L_Z$ for some Cartier divisor on $Z$. 
\end{prop}

\begin{proof}
Taking a   log resolution of $(Z, \Delta_Z)$ dominating $X$, 
we may assume that $f: X\to Z$ is a log resolution of $(Z, \Delta_Z)$. 
In particular, $\Ex(f) = \bigcup_{i \in I} E_i$ is a simple normal crossing divisor. 
Set $E := \sum_{i \in I} E_i$, which is the reduced $f$-exceptional divisor satisfying $\Supp E = \Ex(f)$. 
For a sufficiently small rational number $0 < \epsilon \ll 1$ and $\Delta := (1-\epsilon)E + f_*^{-1}\Delta_Z$, 
we can write 
\[
K_X  + \Delta \equiv f^*(K_Z+\Delta_Z) + \sum_{i \in I} b_i E_i 
\]
for some $b_i \in \Q_{>0}$. 
Note that  $(X, \Delta)$ is klt. 
We run a $(K_X + \Delta )$-MMP over $Z$: 
\[
X =: X_0 \dashrightarrow X_1 \dashrightarrow \cdots \dashrightarrow X_{\ell} =:Y. 
\]

We now prove that the push-forward $L_Y$ of $L$ on $Y$ is Cartier. 
For each $j$, let $\Delta_j$ and $L_j$ be the push-forwards of 
$\Delta$ and $L$ on $X_j$, respectively. 
For the associated birational contraction $X_j \to W_j$, we have the following birational morphisms: 
\[
\begin{tikzcd}
X_j \arrow[rd, "\varphi_j"']\arrow[rr, dashrightarrow] & & X_{j+1}\arrow[ld, "\psi_j"]\\
& W_j
\end{tikzcd}
\]
where $-(K_{X_j} +\Delta_j)$ is $\varphi_j$-ample. 
By Theorem \ref{t nume triv descent}, $L_{j+1}$ is Cartier if so is $L_j$. 
Then $L_Y := L_{\ell}$ is Cartier by induction on $j$. 

Let $g: Y \to Z$ be the induced birational morphism. 
By construction and the negativity lemma,  we have 
$L_Y \equiv_g 0$ and $K_Y+ \Delta_Y \equiv g^*(K_Z+ \Delta_Z) \equiv 0$. 
Again by Theorem \ref{t nume triv descent}, we get $L_Y \sim g^*L_Z$ for some Cartier divisor $L_Z$ on $Z$. 
\end{proof}

\subsection{Decomposition into pl-MMPs}

\begin{lem}\label{l large reduced div1}
Let $Z$ be a projective normal threefold. 
Take a proper closed subset $B$ of $Z$. 
Then there exists a reduced Cartier divisor $H_Z$ on $Z$ satisfying $B \subset H_Z$. 
\end{lem}

\begin{proof}
We may assume that $B$ is irreducible. 
Fix an ample Cartier divisor $A_Z$ on $Z$. 

We first treat the case when $\dim B \leq 1$ (i.e., $B$ is either a point or a curve). 
Take a birational morphism $\alpha : V \to Z$ 
from a smooth projective threefold $V$ such that 
$\dim \alpha^{-1}(b) >0$ for every closed point $b \in B$. 
Moreover, we may assume that 
there exist an effective divisor $E$ on $V$ such that 
$\Supp E \subset \Ex(\alpha)$ and $-E$ is $\alpha$-ample 
\cite[Theorem 1 and Remark 5]{KW24}, \cite[Lemma 3.39]{KM98}. 
Then $m\alpha^*A_Z - E$ is ample for some integer $m>0$, which enables us to find another integer $n>0$ such that $|n(m\alpha^*A_Z - E)|$ is very ample. 
Take a general member $H_V$ of $|n(m\alpha^*A_Z - E)|$. 
By the Bertini theorem, $H_V$ is a prime divisor on $V$. 
We then get 
\[
H_Z := \alpha_*H_V \sim \alpha_*(n(m\alpha^*A_Z - E)) = nmA_Z. 
\]
Hence $H_Z$ is a prime Cartier divisor. 
Since $H_Z$ intersects $\alpha^{-1}(b)$ for every closed point $b \in B$, 
we get $B \subset H_Z$. 
This completes the proof for the case when $\dim B \leq 1$. 

In what follows, we assume that  $B$ is a prime divisor. 
There exist a birational morphism $\beta : W \to Z$ 
and an effective Cartier divisor $F$ on $W$ such that 
$W$ is a smooth projective threefold, 
$\Supp F \subset \Ex(\beta)$, and $-F$ is $\beta$-ample 
\cite[Theorem 1]{KW24}, \cite[Lemma 3.39]{KM98}. 
Then $m\beta^*A_Z -F$ is ample for some integer $m>0$, and hence 
$|n(m\beta^*A_Z -F) - B_W|$ is very ample for the proper transform $B_W$ of $B$ and some integer $n>0$. 
Take a general member $D_W$ of $|n(m\beta^*A_Z -F) - B_W|$, 
which is a prime divisor satsifying $D_W \neq B_W$. 
We then get 
\[
D_Z := \beta_*D_W \sim \beta_*(n(m\beta^*A_Z -F) - B_W) = nm A_Z -B. 
\]
Therefore, $H_Z := D_Z+B$ is a reduced Cartier divisor, as required. 
\end{proof}

\begin{lem}\label{l large reduced div2}
Let $Z$ be a projective normal threefold. 
Fix a proper closed subset $C$ of $Z$
Then there exists a reduced Cartier divisor $H_Z$ on $Z$ such that  
the following property $(\star)$ holds. 
\begin{enumerate}
\item[($\star$)] $C \subset H_Z$ and if $\alpha : V \to Z$ is 
a birational morphism from a projective $\Q$-factorial threefold $V$, 
then the prime divisors contained in $\alpha^{-1}(H_Z) \cup \Ex(\alpha)$ generate $N^1(V/Z)$. 
\end{enumerate}
\end{lem}

\begin{proof}
Fix a birational morphism $\beta : W \to Z$ from a projective $\Q$-factorial threefold $W$. 
It is enough to find a reduced Cartier divisor $H_Z$ on $Z$ satisfying 
the property $(\star)$ for the case when $V= W$ and $\alpha = \beta$. 
Fix prime divisors $D_1, ..., D_r$ which generate $N^1(W/Z)$. 
Then the assertion holds by applying Lemma \ref{l large reduced div1} to $B := C \cup \beta(D_1 \cup \cdots \cup D_r)$. 
\end{proof}

Let us recall results on three-dimensional minimal model program (cf.\ \cite{HX15}, \cite{CTX15}, \cite{BW17}, \cite{HW22}, \cite{Xu24}). 

\begin{thm}\label{t lc mmp}
Assume $p>5$. 
Let $(X, \Delta)$ be a three-dimensional $\Q$-factorial lc pair, 
where $\Delta$ is an $\Q$-divisor. 
Let $f:X \to Z$ be a projective morphism to a quasi-projective scheme $Z$. 
Then the following hold. 
\begin{enumerate}
\item There exists a $(K_X+\Delta)$-MMP over $Z$ which terminates. 
\item 
If $A$ is an $f$-ample $\Q$-Cartier $\Q$-divisor and $K_X+\Delta$ is $f$-pseudo-effective, 
then there exists a $(K_X+\Delta)$-MMP over $Z$ 
scaling of $A$ which terminates. 
\end{enumerate}
\end{thm}

\begin{proof}
The assertion (1) follows from \cite[Theorem 1.1]{HNT20}. 
The assertion (2) holds by 
\cite[Theorem 1.3, Theorem 1.4, Theorem 1.6(2), Theorem 3.2]{HNT20}. 
\end{proof}

\begin{nota}[Decomposition into pl-MMPs]\label{n pl MMP}
Assume $p>5$. 
Let $(X, \Delta)$ be a three-dimensional 
projective $\Q$-factorial klt pair. 
Let $f : X \to Z$ be a birational contraction of a $(K_X+\Delta)$-negative extremal ray. 
Fix a reduced Cartier divisor $H_Z$ on $Z$ as in Lemma \ref{l large reduced div2} for the non-snc locus $C$ of $(Z, f_*\Delta)$. 
Take a log resolution $\mu_Z : V \to Z$ of $(Z, f_*\Delta +H_Z)$ which factors through $X$: 
\[
\mu_Z : V \xrightarrow{\mu_X} X \xrightarrow{f} Z. 
\]
We may assume that $\mu_Z(\Ex(\mu_Z)) \subset  H_Z$. 
Set $B_V := (\mu_Z)_*^{-1}(f_*\Delta)$ and 
$H_V := (\mu_Z)_*^{-1}H_Z$. 
Let $E_V$ be the reduced $\mu_Z$-exceptional divisor on $V$ satisfying $\Supp E_V = \Ex(\mu_Z)$. 

First, we run a $(K_V+H_V+E_V+B_V)$-MMP over $Z$ (Theorem \ref{t lc mmp}): 
\begin{equation}\label{e1 n pl MMP}
V =:V_0 \overset{\varphi_0}{\dashrightarrow} V_1 
\overset{\varphi_1}{\dashrightarrow} \cdots 
\overset{\varphi_{\ell-1}}{\dashrightarrow} V_{\ell} =:V', 
\end{equation}
where $V'$ denotes its end result. 
Let $\psi_i : V_i \to W_i$ be the associated extremal-ray contraction. 
We shall show that $\Ex(\psi_i) \subset S_i$ for some 
prime divisor $S_i \subset \Supp(H_{V_i}+E_{V_i})$ 
(Lemma \ref{l pl MMP1}(1)).

Next, we run a $(K_{V'}+E_{V'}+B_{V'})$-MMP over $Z$ with scaling of $H_{V'}$ (Theorem \ref{t lc mmp}): 
\begin{equation}\label{e2 n pl MMP}
V' =:V'_0 \overset{\varphi'_0}{\dashrightarrow} V'_1 
\overset{\varphi'_1}{\dashrightarrow} \cdots 
\overset{\varphi'_{\ell'-1}}{\dashrightarrow} V'_{\ell'} =:V'', 
\end{equation}
where $V''$ denotes its end result. 
Let $\psi'_i : V'_i \to W'_i$ be the associated extremal-ray contraction. 
We shall show that $\Ex(\psi'_j) \subset S_j$ for some 
prime divisor $S'_j \subset \Supp E_{V'_i}$ 
(Lemma \ref{l pl MMP1}(2)). 
Moreover,  one of the following holds (Lemma \ref{l pl MMP1}(3)(4)): 
\begin{itemize}
\item $f:X \to Z$ is a divisorial contraction and $Z=  V''$. 
\item $f: X \to Z$ is a flipping contraction and 
$V'' \to Z$ is its flip. 
\end{itemize}
\end{nota}

\begin{lem}\label{l pl MMP1}
We use Notation \ref{n pl MMP}. 
Then the following hold. 
\begin{enumerate}
\item There exists a $\psi_i$-anti-ample prime divisor $S_i \subset \Supp(H_{V_i}+E_{V_i})$ for every $0 \leq i\leq \ell-1$. 
\item  There exists a $\psi'_j$-anti-ample prime divisor $S_j \subset \Supp E_{V_j}$ for every $0 \leq j\leq \ell'-1$.
\item If $f : X \to Z$ is a divisorial contraction (i.e., $\dim \Ex(f)=2$), then $Z = V''$. 
\item If $f : X \to Z$ is a flipping contraction 
 (i.e., $\dim \Ex(f)=1$) and $f^+ : X^+ \to Z$ is its flip, then $X^+ = V''$. 
\end{enumerate}
\end{lem}

Although this result is probably well known, 
we include a proof for the sake of completeness 
(cf.\ \cite{Fuj07f}).

\begin{proof}
Let us show (1). 
Since the prime divisors contained in $E_{V_i}+H_{V_i}$ generate 
$N^1(V_i/Z)$, 
a divisor $A_i$ on $V_i$ ample over $Z$ can be written as an $\R$-linear combination of them. 
Hence we get $S_i \cdot R_i \neq 0$ for the extremal ray $R_i$ 
corresponding to $\psi_i$ and some prime divisor $S_i$ contained in $\Supp(E_{V_i}+ H_{V_i})$. 
If $S_i \cdot R_i <0$, then we are done. 
Otherwise (i.e., $S_i \cdot R_i>0$), we can find another prime divisor $T_i \subset \Supp(E_{V_i}+ H_{V_i})$ satisfying 
$T_i \cdot R_i <0$ by using 
$0 \sim_Z \mu^*_Z H_Z = H_V+ \sum_{\lambda} b_{\lambda} E_{V, \lambda}$. 
Here we used the fact that $b_{\lambda}>0$ for every $\mu_Z$-exceptional prime divisor $E_{V, \lambda}$. 
Thus (1) holds.

Let us show (2). 
The second MMP (\ref{e2 n pl MMP}) is $H_{V'}$-positive, as 
it is scaling of $H_{V'}$. 
By $0\equiv_Z \mu_Z^*H_Z = H_V+ \sum_{\lambda} b_{\lambda}E_{\lambda}$, this MMP is 
$\sum_{\lambda} b_{\lambda}E_{\lambda}$-negative. 
Therefore, each step of the MMP  (\ref{e2 n pl MMP}) is negative on some $E_{\lambda}$. 
Thus (2) holds.

We now prove ($\star$) below. 
Note that ($\star$) implies (3), because a birational morphism between projective $\Q$-factorial varieties is automatically an isomorphism. 
\begin{enumerate}
\item[($\star$)]$E_{V''}=0$, i.e., $V'' \to Z$ is small. 
\end{enumerate}
Since $-(K_X+\Delta)$ is ample over $Z$, 
$M_V := -\mu^*_X(K_X+\Delta)$ is semi-ample over $Z$. 
It holds that  
\[
K_V +E_V+B_V \sim_{\Q} \mu^*_X(K_X+\Delta) + \sum_{\lambda \in \Lambda} a_{\lambda} E_{V, \lambda},  
\]
where  we have $a_{\lambda}>0$ for every $\mu_Z$-exceptional prime divisor $E_{\lambda}$ (even if it is the proper transform of an $f$-exceptional prime divisor). 
Therefore, we get 
\[
K_{V''} +E_{V''} +B_{V''} +  M_{V''} \sim_{\Q} \sum_{\lambda \in \Lambda} a_{\lambda} E_{V'', \lambda}. 
\]

Suppose that $E_{V''} \neq 0$, i.e., 
$\sum_{\lambda \in \Lambda} a_{\lambda} E_{V'', \lambda} \neq 0$. 
We run a $(K_{V''} +E_{V''} +B_{V''} +  M_{V''})$-MMP over $Z$: 
\[
V'' =: V''_0 \dashrightarrow V''_1 \dashrightarrow \cdots \dashrightarrow V''_{\ell''} =:V'''. 
\]
By the negativity lemma, we have $\sum_{\lambda} a_{\lambda} E_{V''', \lambda} =0$. 
Hence we can find an index $i$ such that $\varphi''_i : V''_i \to V''_{i+1}$ is a divisorial contraction. 
We have $E_{V''_i, \lambda} = \Ex(\varphi''_i)$ for some $\lambda \in \Lambda$. 
Since each of $M_V$ and $K_{V''}+E_{V''}+B_{V''}$ is semi-ample, 
the sum $K_{V''_i} +E_{V''_i} +  M_{V''_i}$ of their push-forwards on $V''_i$ is movable over $Z$,  i.e., 
the relative stable base locus is of codimension $\geq 2$. 
Then there exists an effective $\Q$-divisor $D$ on $V''_i$ such that 
$K_{V''_i} +E_{V''_i} +B_{V''_i}+  M_{V''_i} \sim_{\Q, Z} D$ and 
$\Ex(\varphi''_i) \not\subset \Supp\,D$. 
Hence we can find a $\varphi''_i$-contracted  curve $\zeta \subset \Ex(\varphi''_i)$ such that 
\begin{enumerate}
\item[(i)] $\varphi''(\zeta)$ is a point, and 
\item[(ii)] 
$\zeta \not\subset \Supp D$. 
\end{enumerate}
Then (i) implies $(K_{V''_i} +E_{V''_i}+B_{V''_i} +  M_{V''_i}) \cdot \zeta <0$, whilst it follows from (ii) that  
$ 0 \leq D \cdot \zeta = (K_{V''_i} +E_{V''_i} +B_{V''_i}+  M_{V''_i}) \cdot \zeta \geq 0$. 
This is absurd. 
This completes the proof of ($\star$), and hence of (3).

Let us show (4). 
Note that we have  $\rho(V''/Z) = \rho(X/Z) = 1$, 
because each of $\rho(V/Z) - \rho(X/Z)$ and $\rho(V/Z)-\rho(V''/Z)$ 
is equal to the number of $\mu_Z$-exceptional prime divisors. 
By ($\star$) and construction, 
\begin{itemize}
\item each of $f: X \to Z$ and $g: V'' \to Z$ is a small birational morphism, 
\item $-(K_X+\Delta)$ is $f$-ample, $(K_{V''}+B_{V''})$ is $g$-nef, and 
\item $B_{V''}$ is the proper transform of $\Delta$. 
\end{itemize}
Hence it is enough to show that  $(K_{V''}+B_{V''})$ is $g$-ample. 
Otherwise, we would have $K_{V''}+B_{V''} \equiv_g 0$, which implies that $K_Z+\Delta_Z$ is $\Q$-Cartier and 
$K_{V''}+B_{V''} \sim_{\Q} g^*(K_Z+\Delta_Z)$ for $\Delta_Z := g_*B_{V''} = f_*\Delta$ (cf.\ Theorem \ref{t nume triv descent}). 
Then $K_X+\Delta \sim_{\Q} f^*(K_Z+\Delta_Z)$, and hence $K_X+\Delta$ is $f$-numerically trivial. 
This is absurd. 
Thus (4) holds. 
\end{proof}

\begin{lem}\label{l nef pl decompo}
We use Notation \ref{n pl MMP}. 
Let $N$ be an $f$-ample nef Cartier divisor on $X$. 
Set $N_V := \mu^*N$ and let $N_{V_i}$ (resp.\ $N_{V'_j}$) 
be its push-forward on $V_i$ (resp.\ $V'_j$). 
Then either 
\begin{enumerate}
\item 
$N_V$ descends to $V_i$ and $N_{V_i}$ is $\psi_i$-ample for some $i$, or 
\item 
$N_V$ descends to $N_{V'_j}$ and 
$N_{V'_j}$ is $\psi'_j$-ample for some $j$. 
\end{enumerate}
\end{lem}

\noindent
We say that $N_V$ {\em descends to} $V_i$ if 
there exists a Cartier divisor $M_{k}$ on $W_{k}$ 
satisfying $N_{V_{k}} \sim \psi_k^*M_k$ for every $0 \leq k \leq i-1$. 
In this case, all $N_{V_0}, ..., N_{V_i}$ are Cartier (Theorem \ref{t nume triv descent}). 
The definition of \lq\lq$N_V$ descends to $N_{V'_j}$" is given in a similar way.

\begin{proof}
Suppose that $N_{V_i}$ is $\psi_i$-numerically trivial for every $i$ and $N_{V'_j}$ is $\psi'_j$-numerically trivial for every $j$. 
It is enough to derive a contradiction. 
In this case, $N_V$ descends to the end result $V''$ (Theorem \ref{t nume triv descent}). 
In particular, $N_V$ is Cartier. 
If $f: X \to Z$ is a divisorial contraction, then we have $Z=V''$, and hence $N_V$ descends to $Z$ and $N_V$ is numerically trivial over $Z$. 
This contradicts the $f$-ampleness of  $N_V$. 

Hence $f: X \to Z$ is a flipping contradiction. 
In this case, $V'' \to Z$ is the flip of $f$. 
Note that $N_{V''}$ is nef, because $N_{V''}$ is obtained by descending from a nef Cartier divisor $N_V$. 
In particular, $N_{V''}$ is nef over $Z$. 
On the other hand, $N_{V''}$ is the proper transform of an $f$-ample divisor $N_X$ via the flip $X \dashrightarrow V''$, and hence $N_{V''}$ is anti-ample over $Z$. 
This is absurd. 
\end{proof}

\section{Pseudo-effectivity of $K_X+5L$}

\begin{prop}\label{p 3L MMP step}
Assume $p>5$. 
Let $(X, \Delta)$ be a three-dimensional projective $\Q$-factorial dlt pair  
and let $L$ be a nef Cartier divisor $L$ on $X$. 
Take a $(K_X+\Delta+3L)$-negative extremal ray $R$ of $\overline{\NE}(X)$ and 
let $f:X \to Z$ be its contraction. 
Assume that $f$ is birational. 
Then there exists a nef Cartier divisor $L_Z$ on $Z$ such that 
$L \sim f^*L_Z$. 
\end{prop}

\begin{proof}
We may assume that $(X, \Delta)$ is klt. 
It is enough to show that $L \equiv_f 0$ (Theorem \ref{t nume triv descent}). 
Suppose that $L \not\equiv_f 0$, i.e., $L$ is $f$-ample. 
Taking a suitable resolution and running MMPs as in Notation \ref{n pl MMP}, 
Lemma \ref{l nef pl decompo} enables us to find 
a three-dimensional $\Q$-factorial projective plt pair $(U, S+B)$, 
a nef Cartier divisor $L_U$ on $U$, 
and a birational contraction $\psi : U \to W$ of a $(K_U+S+B+3L_U)$-negative extremal ray such that 
$-S$ and $L_U$ are $\psi$-ample. 
Since $-S$ is $\psi$-ample, the Stein factorisation 
$\psi_S : S \to T$ of $S \hookrightarrow U \xrightarrow{\psi} W$ is a nontrival contraction. 
This contradicts Lemma \ref{l surface 3L}, because the pair $(S, B_S)$ 
defined by the adjunction $(K_U+S+B)|_S = K_S+B_S$ 
is klt. 
\end{proof}

In the proof of our main result (Theorem \ref{t Fujita}), the problem is reduced to the case when $X$ has a $(K_X+5L)$-Mori fibre space $X \to B$. 
We separately treat the two cases $\dim B > 0$ (Lemma \ref{l easy MFS})  and $\dim B =0$ (Proposition \ref{p key Fujita}).

\begin{lem}\label{l easy MFS}
Let $(X, \Delta)$ be a three-dimensional
projective $\Q$-factorial  dlt pair and 
let $f: X \to B$ be a $(K_X+\Delta)$-negative Mori fibre space with $\dim B> 0$. 
Take a nef and big Cartier divisor $L$ on $X$. 
Then $K_X+\Delta+3L$ is $f$-nef. 
\end{lem}

\begin{proof}
We may assume that $(X, \Delta)$ is klt. 
Suppose that $K_X+3L$ is not $f$-nef. 
Let us derive a contradiction. 
By $\rho(X/B)=1$, $K_X+3L$ is $f$-anti-ample. 
Let $X_{\eta}$ be the generic fibre of $f: X \to B$. 
Then both $L_{\eta}$ and $-(K_{X_{\eta}} + 3L_{\eta})$ are ample. 
If $\dim B =2$, then $X_{\eta}$ is a regular conic, which leads to the following contradiction: 
\[
0 > \deg(K_{X_{\eta}} + 3L_{\eta}) \geq -2 + 3 \cdot 1 = 1. 
\]
Hence we get $\dim B =1$. 
In this case, we get a contradiction by 
Proposition \ref{p surface Fujita imperf}. 
\end{proof}

\begin{prop}\label{p key Fujita}
Assume $p>5$. 
Let $X$ be a projective $\Q$-factorial terminal threefold such that 
$-K_X$ is ample and $\rho(X)=1$. 
Let $L$ be an ample Cartier divisor on $X$. 
Then $K_X+5L$ is nef. 
\end{prop}

\begin{proof}
Suppose that $K_X+5L$ is not nef. 
Let us derive a contradiction. 
Set $A :=-(K_X+5L)$, which is an ample Weil divisor. 
Fix a smooth closed point $P \in X$. 

\setcounter{step}{0}

\begin{step}\label{s1 key Fujita}
There exists an effective $\Q$-divisor $D$ on $X$ such that 
$D \sim_{\Q} 2L+ (1/2)A$  and one of (I)--(III) holds. 
\begin{enumerate}
\item[(I)] $(X, D)$ is klt and $\mult_PD >2$. 
\item[(II)] $(X, D)$ is lc but not dlt.
\item[(III)] $(X, D)$ is dlt but not klt. 
\end{enumerate}
\end{step}

\begin{proof}[Proof of Step \ref{s1 key Fujita}]
By the proof of Shokurov's non-vanishing theorem \cite[\S 3.5]{KM98}, 
there exists  an effective $\Q$-divisor $D_1$ on $X$ such that 
$(2L+ (1/2)A) \sim_{\Q} D_1$ and $\mult_PD_1 >2$. 
If $(X, D_1)$ is lc, then one of (I)--(III) holds by setting $D:= D_1$. 
Hence we may assume that $(X, D_1)$ is not lc. 
Let $\lambda$ be the log canonical threshold of $(X, D_1)$. 
Hence $\lambda$ is the rational number such that $0 < \lambda < 1$ and 
$(X, \lambda D_1)$ is lc but not klt. 
By \cite[Theorem 1]{Tan17}, we can find an effective $\Q$-divisor $D_2$ such that $D_2 \sim_{\Q} (1-\lambda)(2L+ (1/2)A)$ and $(X, \lambda D_1 + D_2)$ is lc but not klt. 
Then (II) or (III) holds by setting $D := \lambda D_1 + D_2$. 
This completes the proof of Step \ref{s1 key Fujita}. 
\end{proof}

\begin{step}\label{s2 key Fujita}
(III) does not hold.  
\end{step}

\begin{proof}[Proof of Step \ref{s2 key Fujita}]
Suppose that (III) holds. 
After a suitable perturbation, there exist 
an effective $\Q$-divisor $D'$ on $X$ and a rational number $0 <t <1$ 
such that $2L + tA \equiv D'$, $(X, D')$ is plt but not klt, and 
$S := \rdown{D'}$ is a prime divisor. 
We have $D = S+B$ for the fractional part $B := \{D\}$ of $D$. 
Let $B_S$ be the effective $\Q$-divisor on $S$ 
defined by the adjunction $(K_X+S+B)|_S = K_S + B_S$ \cite{Kol13}. 
Then $(S, B_S)$ is klt. 
It follows from  $2L + tA \equiv D' =S+B$  that  
\[
0 \equiv (K_X+5L+A)|_S \equiv 
(K_X + S+ B + 3L + (1-t)A)|_S \equiv 
K_S +B_S+ (3L +(1-t)A)|_S, 
\]
which contradicts Lemma \ref{l surface 3L}. 
This completes the proof of Step \ref{s2 key Fujita}.  
\end{proof}

\begin{step}\label{s3 key Fujita}
If  (I) or (II) holds, then 
there exists a birational morphism $f : X' \to X$ and 
an effective $\Q$-divisor $\Delta'$ on $X'$ such that 
\begin{enumerate}
\item $(X', \Delta')$ is a three-dimensional projective $\Q$-factorial dlt pair, 
\item $\rho(X')>1$, and 
\item $K_{X'}+ \Delta'+ sf^*L \equiv 0$ for some rational number $s >3$. 
\end{enumerate}
\end{step}

\begin{proof}[Proof of Step \ref{s3 key Fujita}]
Assume that (I) holds. 
Let $f: X' \to X$ be the blowup at $P$. 
Clearly (2) holds. 
Let $\Delta'$ be the $\Q$-divisor defined by 
\[
K_{X'}+ \Delta' =  f^*(K_X+D). 
\]
By $\mult_P D >2$, $\Delta'$ is effective. 
Since $(X, D)$ is klt, (1) holds. 
It holds that  
\[
K_{X'} + \Delta'  + f^*( 3L + (1/2)A)  \equiv 
f^*(K_X+ 5L+A) \equiv 0. 
\]
Thus (3) holds. 
This completes the proof of Step \ref{s3 key Fujita} for the case when (I) holds.

Assume that (II) holds. 
By \cite[Corollary 3.6]{HNT20}, 
there exists a birational morphism $f : X' \to X$ such that 
$X'$ is a projective $\Q$-factorial threefold, 
the $\Q$-divisor $\Delta'$ defined by $K_{X'}+\Delta' = f^*(K_X+D)$ is effective, and $(X', \Delta')$ is dlt. 
In particular, (1) and (2) holds. 
The remaining property (3) holds by 
\[
K_{X'}+ \Delta'+ f^*( 3L + (1/2)A) \equiv 
f^*(K_X + D + 3L + (1/2)A) \equiv 0. 
\]
This completes the proof of Step \ref{s3 key Fujita} for the case when (II) holds. 
\qedhere
\end{proof}

\begin{step}\label{s4 key Fujita}
Neither (I) nor (II) holds. 
\end{step}

\begin{proof}[Proof of Step \ref{s4 key Fujita}]
Suppose that one of (I) and (II) holds. 
Let $f: X' \to X$ and $\Delta'$ be as in Step \ref{s3 key Fujita}. 
By $K_{X'}+\Delta' +3L \equiv -(s-s)f^*L$, there is a 
$(K_{X'}+\Delta'+3L)$-negative extremal ray $R$ of $\overline{\NE}(X')$. 
We have $f^*L \cdot R>0$. 
Let $g: X' \to X''$ be its contraction. 
By $\rho(X')>\rho(X)=1$, we get $\dim X'' >0$. 
If $\dim X'' =3$ (resp.\ $\dim X''<3$), then we would get a contradiction 
by Proposition \ref{p 3L MMP step} (resp.\ Lemma \ref{l easy MFS}). 
This completes the proof of Step \ref{s4 key Fujita}. 
\end{proof}

We get a contradiction by Step \ref{s1 key Fujita}, Step \ref{s2 key Fujita}, and Step \ref{s4 key Fujita}. 
This compltes the proof of Proposition \ref{p key Fujita}. 
\end{proof}

We are ready to prove the main result of this section. 

\begin{thm}\label{t Fujita imperf}\label{t Fujita}
Let $\kappa$ be a field of characteristic $p>5$. 
Let $X$ be a projective normal threefold over $\kappa$ such that $K_X$ is $\Q$-Cartier. 
Take a nef and big Cartier divisor $L$. 
Then $K_X+5L$ is pseudo-effective.
\end{thm}

\begin{proof}
The proof consists of two steps. 

\setcounter{step}{0}

\begin{step}\label{s1 Fujita}
The assertion of Theorem \ref{t Fujita imperf} holds if $\kappa$ is algebraically closed. 
\end{step}

\begin{proof}[Proof of Step \ref{s1 Fujita}]
Replacing $X$ by a relatively minimal model over $X$, we may assume that $X$ is $\Q$-factorial and terminal. 
Suppose that $K_X+5L$ is not pseudo-effective. 
It suffices to derive a contradiction. 
Running  a $(K_X+5L)$-MMP, we may assume that 
there exists a $(K_X+5L)$-Mori fibre space $\pi: X \to B$ 
(Proposition \ref{p 3L MMP step}). 
Then the assertion holds by Lemma \ref{l easy MFS} ($\dim B>0$) and 
Proposition \ref{p key Fujita} ($\dim B=0$). 
This completes the proof of Step \ref{s1 Fujita}. 
\end{proof}

\begin{step}\label{s2 Fujita}
The assertion of Theorem \ref{t Fujita imperf} holds 
without any additional assumptions. 
\end{step}

\begin{proof}[Proof of Step \ref{s2 Fujita}]
Taking a desingularisation of $X$, we may assume that $X$ is regular. 
Replacing $\kappa$ by $H^0(X, \MO_X)$, the problem is reduced to the case when $H^0(X, \MO_X)= \kappa$. 
For the algebraic closure $\overline{\kappa}$ of $\kappa$, we set $X_{\overline{\kappa}} := X \times_{\Spec \kappa} \Spec \overline{\kappa}$. 
Take the reduced structure $(X_{\overline{\kappa}})_{\red}$ of $X_{\ol{\kappa}}$ and its normalisation $Y:=(X_{\overline{\kappa}})^N_{\red}$. 
Note that 
$Y$ is a projective $\Q$-factorial threefold over $\overline{\kappa}$ 
\cite[Lemma 2.2, Lemma 2.5]{Tan18b}. 
For the induced composite morphism 
\[
f :Y=(X_{\overline{\kappa}})^N_{\red} \to (X_{\overline{\kappa}})_{\red}\to X, 
\]
we have 
\[
K_Y + C \sim f^*K_X
\]
for some effective Weil divisor $C$ on $Y$ \cite[Theorem 1.1]{Tan18b}. 
Since the divisor $K_Y + C+ 5f^*L \sim f^*(K_X+5L)$ is pseudo-effective 
(Step \ref{s1 Fujita}), 
so is $K_X+ 5L$. 
This completes the proof of Step \ref{s2 Fujita}. 
\end{proof}
Step \ref{s2 Fujita} completes the proof of Theorem \ref{t Fujita}. 
\end{proof}

\section{Mumford vanishing}

\subsection{Big case}

\begin{lem}\label{l H^1 easy}
Let $\kappa$ be a field. 
Let $X$ be a projective normal threefold over $\kappa$. 
Take a nef  Cartier divisor $L$ on $X$ such that $\nu(X, L)\geq 2$.  
Then there exists a positive integer $m_0$ such that 
$H^1(X, -mL)=0$ for every integer $m \geq m_0$. 
\end{lem}

Although this result should be well known, we give a proof for the sake of completeness. 

\begin{proof}
Taking a desingularisation of $X$, we may assume that $X$ is regular. 
Take a general hyperplane section $H$ of $X$, which is a regular prime divisor on $X$. 
We have an exact sequence 
\[
H^1(X, K_H+mL|_H)  \to H^2(X, K_X+ mL) \to H^2(X, K_X+H+mL). 
\]
By the Fujita vanishing theorem, we may assume that $H^2(X, K_X+H+mL) =0$ for every integer $m \geq 0$. 
As $L|_H$ is nef and big, we have $H^1(X, K_H+mL|_H)=0$ for $m \gg 0$ 
\cite[Theorem 3.8]{Tan18m}. 
\end{proof}

\begin{thm}\label{t H^1 nefbig}
Let $\kappa$ be a field of characteristic $p>5$. 
Let $X$ be a projective normal threefold over $\kappa$ such that 
$K_X$ is $\Q$-Cartier and not big. 
Take a nef and big Cartier divisor $L$ on $X$.  
Then $H^1(X, -L)=0$. 
\end{thm}

\begin{proof}
Taking a  desingularisation of $X$, we may assume that $X$ is regular. 
Suppose that $H^1(X, -L) \neq 0$.  
It holds that $H^1(X, -mL)=0$ for $m \gg 0$ (Lemma \ref{l H^1 easy}). 
Replacing $L$ by $p^eL$ for some $e\geq 0$, 
we may assume that $H^1(X, -pL)=0$. For $\mathcal L:=\MO_X(L)$, we can find a non-trivial $\alpha_{\mathcal L^{-1}}$-torsor 
\[
\varphi: X' \to X 
\]
which satisfies the following properties \cite[Proposition 1.1.2]{Mad16}: 
\begin{itemize}
\item $X'$ is a projective Gorenstein threefold, 
\item $\varphi$ is a finite inseparable surjective morphism of degree $p$, and 
\item $K_{X'} \sim \varphi^*(K_X -(p-1)L)$ (more precisely, 
$\omega_{X'} \simeq \varphi^*\MO_X(K_X -(p-1)L)$). 
\end{itemize}
For the normalisation $\nu : X'' \to X'$ of $X'$, we have 
$K_{X''} +D \sim \nu^*K_{X'}$, where 
$D$ is the conductor divisor, which is an effective Weil divisor on $X''$. 
For the induced composite morphism 
\[
\psi: X'' \xrightarrow{\nu} X' \xrightarrow{\varphi} X, 
\]
we get 
\[
K_{X''} +D+ 5\psi^*L +(p-6)\psi^*L  \sim \nu^*(K_{X'} +(p-1) \varphi^* L) 
\sim \nu^*\varphi^*K_X =\psi^*K_X. 
\]
Since $K_{X''} + 5\psi^*L$ is pseudo-effective  (Theorem \ref{t Fujita imperf}), 
it follows from $p-6>0$ that $K_{X''} +D+ 5\psi^*L +(p-6)\psi^*L$ is big. 
Then $\psi^*K_X$ is big, and hence so is $K_X$. 
This is absurd. 
\qedhere
\end{proof}

\begin{thm}\label{t Mukai p=2}
Assume $p=2$. 
Then there exist a smooth projective threefold $X$ 
and an ample Cartier divisor $L$ on $X$ such that 
$\kappa(X, K_X)=2$, $K_X$ is semi-ample, and $H^1(X, -L) \neq 0$. 
\end{thm}

\begin{proof}
For every integer $n>0$, let $(X_n, D_n, df_n)$ be 
the TR triple as in \cite[Subsection 2.4]{Muk13}. 
As explained there, $X_n$ is an $n$-dimensional smooth projective variety such that $H^1(X_n, -L_n) \neq 0$ for some ample Cartier divisor $L_n$ on $X_n$. 
For the definition of TR triples, see 
\cite[the beginning of Section 2]{Muk13}. 
Set $X :=X_3$ and $L := L_3$. 
Note that we have a surjective morphism $\pi : X  =X_3 \to X_2$ (indeed, $X_{n+1}$ is obtained as a suitable cyclic cover of a $\P^1$-bundle over $X_n$ \cite[the beginning of Subsection 2.1]{Muk13}). 
By \cite[Proposition 2.7]{Muk13}, 
$K_X =K_{X_3}$ is the pullback of an ample Cartier divisor on $X_2$. 
Therefore, $K_X$ is semi-ample and $\kappa(X, K_X) =2$. 
\end{proof}

\subsection{Case $\nu(L)=2$}

\begin{thm}\label{t nu=2 Mumford}
Assume $p>5$. 
Let $X$ be a projective klt threefold such that $-K_X$ is nef.  
Take a nef Cartier divisor $L$ such that $\nu(X, L)=2$. 
Then $H^1(X, -L)=0$. 
\end{thm}

\begin{proof}
Suppose that $H^1(X, -L) \neq 0$. 
After replacing $L$ by $p^eL$ for some integer $e \geq 0$, 
there exists a finite inseparable surjective morphism $\psi : Y \to X$ of degree $p$ from a projective normal threefold $Y$ such that  
\[
K_Y + D + (p-1) L_Y \sim \psi^*K_X, 
\]
where $L_Y := \psi^*L$ and $D$ is an effective Weil divisor on $Y$ 
(cf.\ the proof of Theorem \ref{t H^1 nefbig}).

\setcounter{step}{0}

\begin{step}\label{s1 nu=2 Mumford}
$K_Y  +5\psi^*L$ is not pseudo-effective. 
\end{step}

\begin{proof}[Proof of Step \ref{s1 nu=2 Mumford}]
If $K_Y+5\psi^*L$ is pseudo-effective, then  we would get the following contradiction for an ample Cartier divisor $A_Y$ on $Y$: 
\[
0 = (K_Y + D + (p-1) \psi^*L - \psi^*K_X) \cdot A_Y^2 \geq (p-6)\psi^*L  \cdot A_Y^2 >0. 
\]
This completes the proof of Step \ref{s1 nu=2 Mumford}. 
\end{proof}

\begin{step}\label{s2 nu=2 Mumford}
$L$ is EWM, and hence so is its pullback $\psi^*L =L_Y$. 
For the contraction $\pi_Y: Y \to S_Y$ to a proper algebraic space $S_Y$ induced by $L_Y$,  
the geometric generic fibre $-Y_{\overline{\eta}}$ of $\pi_Y$ is isomorphic to the projective line $\P^1_{\overline{\eta}}$. 
\end{step}

\begin{proof}[Proof of Step \ref{s2 nu=2 Mumford}]
Take a desingularisation $\mu : Y' \to Y$ of $Y$ and set $L_{Y'} := \mu^*L_Y$. 
Since $K_Y + 5L_Y$ is not pseudo-effective (Step \ref{s1 nu=2 Mumford}), neither is $K_{Y'}+ 5L_{Y'}$. 
We run a $(K_{Y'}+5L_{Y'})$-MMP 
\[
Y' =: Y'_0 \dashrightarrow Y'_1 \dashrightarrow \cdots \dashrightarrow Y'_{\ell} =:Y'', 
\]
where $Y''$ denotes its end result. 
This sequence is a $K_{Y'}$-MMP which is $L_{Y'}$-numerically trivial (Proposition \ref{p 3L MMP step}). 
Hence $L_{Y'}$ descends to its push-forward $L_{Y''}$ on $Y''$. 
In particular, $L_{Y''}$ is nef, $\nu(Y'', L_{Y''})=2$, and 
$\alpha^* L_{Y'} \sim \beta^*L_{Y''}$ for birational morphisms $\alpha: Y''' \to Y'$ and $\beta : Y''' \to Y''$ compatible with $Y' \dashrightarrow Y''$. 
Since $K_{Y'}+5L_{Y'}$ is not pseudo-effective, 
there is a $(K_{Y'}+5L_{Y'})$-negative Mori fibre space $\rho: Y'' \to T$. 
Since we have $L_{Y''} \sim \rho^*L_T$ for some nef Cartier divisor $L_T$ on $T$ and $\nu(T, L_T) =\nu(Y'', L_{Y''})=\nu(X, L)=2$, 
we get $\dim T \geq 2$. 
On the other hand, we have $\dim T < \dim Y'' =3$, as $\rho : Y'' \to T$ is a Mori fibre space. 
Therefore, $\dim T = 2$, and hence $L_T$ is a nef and big Cartier divisor on a projective normal surface $T$. 
Hence $L_T$ is EWM \cite[Theorem 1.9]{Kee99}. 
Then its pullback $\beta^*L_{Y''}$ on $Y'''$ is EWM  \cite[Definition-Lemma 1.0(4)]{Kee99}. 
Hence $L_Y$ is EWM, and hence so is $L$ \cite[Lemma 1.5]{Kee99}.

Let $\pi_{Y}: Y \to S_Y$ and $\pi_{Y''} : Y'' \to S_{Y''}$ be the contractions to two-dimensional proper normal integral algebraic spaces induced by $L_Y$ and $L_{Y''}$, respectively. 
Then the generic fibres of $\pi_{Y}, \pi_{Y''}$, and $\rho: Y'' \to T$ coincide. 
Hence { $-K_{Y_{\xi}}$ is ample, where $\xi\in S_Y$ is the generic point}. 
By $p>2$, we get $Y_{\overline{\xi}} \simeq \P^1_{\overline{\xi}}$.  
This completes the proof of Step \ref{s2 nu=2 Mumford}. 
\end{proof}

Let $\pi:X \to S$ be the contraction to two-dimensional proper normal irreducible algebraic spaces induced by $L$. 
Note that $S$ becomes a normal quasi-projective surface after removing finitely many closed points of $S$. 
In particular, 
for the generic point $\eta$ of $S$, 
the generic fibres $X_{\eta}$ of $\pi$ can be defined scheme-theoretically.

\begin{step}\label{s3 nu=2 Mumford}
$-K_{X_{\eta}}$ is ample. 
\end{step}

\begin{proof}[Proof of Step \ref{s3 nu=2 Mumford}]
We now show that $-K_{X_{\eta}}$ is ample.
Since $-K_X$ is nef, so is $-K_{X_\eta}$. 
Suppose that  $K_{X_\eta} \equiv 0$. 
It suffices to derive a contradiction. 
Then general fibres of $X \to T$ are smooth elliptic curves by $p>3$ 
\cite[Corollary 1.8]{PW22}. 
Since $X \to T$ and $Y \to T_Y$ coincide up to a universal homeomorphism. 
Hence there is a universal homeomorphism $Y_{t'} \to X_t$ between general fibres. 
This is a contradiction, because  
$Y_{t'} = \P^1$ and $X_t$ is a smooth elliptic curve. 
This completes the proof of Step \ref{s3 nu=2 Mumford}. 
\end{proof}

\begin{step}\label{s4 nu=2 Mumford}
$L$ is not semi-ample. 
\end{step}

\begin{proof}[Proof of Step \ref{s4 nu=2 Mumford}]
Suppose that $L$ is semi-ample. 
In this case, $S$ is a projective normal surface. 
Then $-K_X$ is $\pi$-nef and $\pi$-big (Step \ref{s3 nu=2 Mumford}). 
We then get $L \sim \pi^*L_S$ for some nef and big Cartier divisor $L_S$ on $S$ (Theorem \ref{t nume triv descent}). 
Moreover, it follows from $R^1\pi_*\MO_X=0$ \cite[Theorem 25]{BK23} that 
\[
R^1\pi_*\MO_X(-L) \simeq R^1\pi_*\MO_X \otimes \MO_S(-L_S) =0. 
\]
By the exact sequence 
\[
0 \to H^1(S, -L_S) \to H^1(X, -L) \to H^0(S, R^1\pi_*\MO_X(-L)) =0
\]
obtained from the corresponding Leray spectral sequence, 
it suffices to show that $H^1(S, -L_S)=0$. 
It follows from  \cite[Corollary 4.10(1)]{Eji19} that $-K_S$ is pseudo-effective. 
Taking the minimal resolution $\mu : S' \to S$ of $S$, 
$-K_{S'}$ is pseudo-effective, which implies 
$\kappa(S', K_{S'}) \leq 0$ (cf.\ \cite[Lemma 5.4]{GLP15}). 
Therefore, $H^1(S, -L_S) \hookrightarrow H^1(S', -\mu^*L_{S})=0$. 
This completes the proof of Step \ref{s4 nu=2 Mumford}. 
\end{proof}

\begin{step}\label{s5 nu=2 Mumford}
$L$ is semi-ample. 
\end{step}

\begin{proof}[Proof of Step \ref{s5 nu=2 Mumford}]
Fix an \'etale surjective morphism $\wt{S} \to S$ from a finite disjoint union 
$\wt{S}$ of affine normal surfaces. 
For the base change 
\[
\wt{\pi} : \wt{X} := X \times_S \wt{S} \to \wt{S}, 
\]
we see that $-K_{\wt{X}}$ is $\wt{\pi}$-nef and $\wt{\pi}$-big (Step \ref{s3 nu=2 Mumford}). 
For $L_{\wt{X}} := \pr_1^*L$, 
there exists a Cartier divisor $L_{\wt{S}}$ on $\wt{S}$ 
satisfying $L_{\wt{X}} \sim \wt{\pi}^*L_{\wt{S}}$. 
In particular, $L|_{X_s}$ is semi-ample for every point $s \in S$. 
By \cite[Theorem 4.6]{CT20}, we get $L^{\otimes n} \simeq \pi^*L_S$ 
for some $n >0$ and an invertible sheaf $L_S$ on $S$. 
Then $L$ is semi-ample by \cite[Definition-Lemma 1.0(5)]{Kee99}. 
This completes the proof of Step \ref{s5 nu=2 Mumford}. 
\end{proof}
We obtain a contradiction by Step \ref{s4 nu=2 Mumford} and Step \ref{s5 nu=2 Mumford}. 
This completes the proof of Theorem \ref{t nu=2 Mumford}. 
\end{proof}

\begin{thm}\label{t nu=2 Mumford2}
Assume $p>5$. 
Let $(X, \Delta)$ be a three-dimensional projective klt pair 
such that $-(K_X+\Delta)$ is nef and big and every nonzero coefficient $c$ of $\Delta$ satisfies $c > 2/p$.   
Take a nef Cartier divisor $L$ such that $\nu(X, L)=2$. 
Then $H^1(X, -L)=0$. 
\end{thm}

\begin{proof}
The same argument as in Theorem \ref{t nu=2 Mumford} works 
after replacing $K_X$ by $K_X+\Delta$ suitably. 
Since every nonzero coefficient $c$ of $\Delta$ satisfies $c > 2/p$, 
we can apply 
\cite[Corollary 4.10(1)]{Eji19} in the proof of Step \ref{s4 nu=2 Mumford}. 
\end{proof}

\begin{ex}\label{e Mumford fail any p}
Take a smooth projectuve curve $C$. 
Let $S$ be a smooth projective surface such that there exists an ample Cartier divisor $L_S$ on $S$ satisfying $H^1(S, -L_S) \neq 0$. 
Set $X:= S \times C$ and $L:= \pr_1^*L_S$. 
Then $L$ is a semi-ample Cartier divisor on $X$ such that 
$\nu(X, L) =2$ and $H^1(X, -L) \neq 0$, where 
the latter condition is insured by the injection 
\[
0 \neq H^1(S, -L_S) \hookrightarrow H^1(X, -L)
\]
induced by the corresponding Leray spectral sequence. 
Set $\kappa :=\kappa(X, K_X), \kappa_S := \kappa(S, K_S)$, and 
$\kappa_C:= \kappa(C, K_C)$. 
\begin{enumerate}
\item 
We can find a pair $(S, L_S)$ satisfying the above properties 
and $\kappa_S = 2$. 
Corresponding to $\kappa_C = -\infty, 0, 1$, 
we find an example satisifying $\kappa =-\infty, 2, 3$. 
\item 
Assume that $p =2$ or $p=3$. 
Then we can find a pair $(S, L_S)$ satisfying the above properties 
and $\kappa_S = 1$. 
Corresponding to $\kappa_C = -\infty, 0, 1$, 
we find an example satisfying $\kappa =-\infty, 1, 2$. 
\end{enumerate}
\end{ex}



\subsection{Weak Fano threefolds}
For a smooth Fano threefold $X$ in positive characteristic, 
it is known that $H^{>0}(X, \MO_X)=0$ \cite[Theorem 2.4]{FanoI}
(see also \cite[Corollary 1.5]{SB97} and \cite[Corollary 3.7]{Kaw2}).
The following result is a generalisation of this result to  weak Fano threefolds in characteristic $p>5$. 

\begin{thm}\label{t irreg}
Assume $p>5$. 
Let $X$ be a Gorenstein canonical projective threefold 
such that $-K_X$ is nef and big. 
Then the following hold. 
\begin{enumerate}
\item 
$H^i(X, \MO_X)=0$ for every $i>0$. 
\item $X$ is simply connected, i.e., if $\pi: Y \to X$ is a finite \'etale morphism from a normal threefold $Y$, then $\pi$ is an isomorphism. 
\item 
$\Pic\,X \simeq \Z^{\oplus \rho(X)}$. 
\end{enumerate}
\end{thm}

\begin{proof}
Let us show (1). 
Since $X$ has rational singularities \cite[Corollary 1.3]{ABL22}, 
we may assume that $X$ is smooth. 
By Serre duality and $H^0(X, K_X)=H^1(X, K_X)=0$ (Theorem \ref{t H^1 nefbig}), 
we get $H^2(X, \MO_X) = H^3(X, \MO_X)=0$. 
Since $X$ is RCC \cite[Theorem 1.5]{GNT19}, 
we get $h^1(X, \MO_X) \leq h^2(X, \MO_X)=0$. 
Thus (1) holds. 

Let us show (2). 
Take a finite \'etale morphism $\pi: Y \to X$ from a normal threefold $Y$. 
Then $Y$ is also a Gorenstein canonical weak Fano threefold. 
We then get 
\[
1 \overset{{\rm (1)}}{=} \chi(Y, \MO_Y) = (\deg \pi) \cdot \chi(X, \MO_X)  \overset{{\rm (1)}}{=} \deg \pi. 
\]
Thus (2) holds. 

Let us show (3). 
Take a numerically trivial Cartier divisor $N$ on $X$. 
It is enough to show $H^0(X, N) \neq 0$. 
We have $h^2(X, N) = h^1(X, K_X -N) = 0$  (Theorem \ref{t H^1 nefbig}). 
Then the Riemann--Roch theorem implies 
\[
h^0(X, N) \geq \chi(X, N) = \chi(X, \MO_X) =1. 
\]
Thus (3) holds. 
\end{proof}

\bibliographystyle{skalpha}
\bibliography{bibliography.bib}

\end{document}